\documentclass[a4paper, 11pt]{amsart}
\usepackage[utf8]{inputenc}
\usepackage[english]{babel}
\usepackage[T1]{fontenc}
\usepackage{amsmath, amssymb, amsfonts}
\usepackage{enumerate}
\usepackage{graphicx}
\usepackage{comment}
\usepackage[hmargin=1.5cm, vmargin=2.5cm]{geometry}
\usepackage{tikz}
\usepackage[all,cmtip]{xy}
\usepackage[colorlinks=true, citecolor=blue]{hyperref}

\title{Torsion and K-theory for some free wreath products}
\author{Amaury Freslon}
\author{Rub\'{e}n Martos}
\address{A. Freslon, Laboratoire de Math\'ematiques d'Orsay, Univ. Paris-Sud, CNRS, Universit\'e Paris-Saclay, 91405 Orsay, France}
\address{R. Martos, Univ. Paris Diderot, Sorbonne Paris Cité, IMJ-PRG, UMR 7586 CNRS, Sorbonne Universit\'es, UPMC Univ. Paris 06, F-75013, Paris, France}
\email{amaury.freslon@math.u-psud.fr, ruben.martos@imj-prg.fr}
\date{}

\theoremstyle{plain}
\newtheorem{thm}{Theorem}[section]
\newtheorem{prop}[thm]{Proposition}
\newtheorem{cor}[thm]{Corollary}
\newtheorem{lem}[thm]{Lemma}

\theoremstyle{definition}
\newtheorem{de}[thm]{Definition}
\newtheorem{ex}[thm]{Example}

\theoremstyle{remark}
\newtheorem{rem}[thm]{Remark}

\DeclareMathOperator{\id}{id}
\DeclareMathOperator{\Ind}{Ind}
\DeclareMathOperator{\Irr}{Irr}
\DeclareMathOperator{\Rep}{Rep}
\DeclareMathOperator{\Res}{Res}
\DeclareMathOperator{\Span}{span}
\DeclareMathOperator{\Stab}{Stab}

\newcommand{\C}{\mathbb{C}}
\newcommand{\CC}{\mathcal{C}}
\newcommand{\CI}{\mathcal{CI}}
\newcommand{\D}{\Delta}
\newcommand{\F}{\mathbf{F}}
\newcommand{\G}{\mathbb{G}}

\newcommand{\HH}{\mathbb{H}}

\newcommand{\N}{\mathbb{N}}
\newcommand{\R}{\mathbb{R}}
\newcommand{\T}{\mathcal{T}}
\newcommand{\Z}{\mathbb{Z}}

\begin{document}

\begin{abstract}
We classify torsion actions of free wreath products of arbitrary compact quantum groups by $S_{N}^{+}$ and use this to prove that if $\G$ is a torsion-free compact quantum group satisfying the strong Baum-Connes property, then $\G\wr_{\ast}S_{N}^{+}$ also satisfies the strong Baum-Connes property. We then compute the K-theory of free wreath products of classical and quantum free groups by $SO_{q}(3)$.
\end{abstract}

\maketitle

\section{Introduction}

This paper is concerned with K-theory for C*-algebras associated to discrete quantum groups. Since the fundamental work of M. Pimsner and D.-V. Voiculescu \cite{pimsner1982k} proving that $C^{*}_{r}(\F_{n})\ncong C^{*}_{r}(\F_{m})$ if $n\neq m$ as a corollary of the computation of their $K_{1}$-groups, the computation of K-theory groups of reduced C*-algebra has been an important subject. A very general method to carry these computations by geometric means was developed under the name of \emph{Baum-Connes conjecture} (see \cite{baum1994classifying}). The idea is to build an \emph{assembly map} between the K-theory on the right-hand side and the K-homology of a toplogical space associated to the group on the left-hand side. The conjecture is that this map is an isomorphism. Even though the strongest forms of the conjecture are known to fail \cite{higson2002counterexamples}, it is still a powerful method for a large class of groups, for instance those satisfying the Haagerup property.

If we turn to discrete quantum groups, the very definition of the assembly map and the statement of a proper Baum-Connes conjecture is unclear. D. Goswami and A. Kuku gave in \cite{goswami2003complete} a non-commutative version of the classifying space for proper actions of a discrete quantum group together with an assembly map. However, there is no method available to prove that the assembly map is bijective or to compute the K-homology of the left-hand side. Forgetting the geometric aspects of the problem and focusing on the homological algebra, R. Meyer and R. Nest gave in \cite{meyer2006baum} a version of the Baum-Connes conjecture in terms of the triangulated structure of the equivariant Kasparov category. Using this, they were able to prove it for duals of compact Lie groups.

The first application of this method to genuine quantum groups is due to C. Voigt in \cite{voigt2011baum}, where he proved the Baum-Connes conjecture for free orthogonal quantum groups and computed their K-theory. Later, he adapted with R. Vergnioux ideas of G. Kasparov an G. Skandalis \cite{kasparov1991groups} to prove in \cite{vergnioux2013k} the Baum-Connes conjecture for arbitrary free products of free unitary and orthogonal quantum groups and compute their K-theory. All these cases have an important feature : they are \emph{torsion-free} in a K-theoretic sense. In the presence of torsion, things become more complicated and the only results so far are also due to C. Voigt \cite{voigt2015structure} for quantum automorphism groups of finite-dimensional C*-algebras.

Our goal in the present paper is to investigate the Baum-Connes conjecture for the free wreath product of an arbitrary compact quantum group $\G$ by the quantum permutation group $S_{N}^{+}$, as defined by J. Bichon in \cite{bichon2004free}. This study naturally splits into two parts :
\begin{itemize}
\item The first one is the computation of the torsion of the corresponding discrete quantum groups. This is a difficult problem in general, but Y. Arano and K. De Commer introduced in \cite{arano2015torsion} a powerful method to show that a quantum group is torsion-free. Free wreath products are never torsion-free, but we will see that their method can still give enough information on the torsion to yield a complete classification. Following the ideas of \cite{voigt2015structure}, it is then easy to deduce that $\G\wr_{\ast}S_{N}^{+}$ has the so-called \emph{strong Baum-Connes property} if $\G$ is a torsion-free compact quantum group already satisfying the strong Baum-Connes property.
\item The strong Baum-Connes property means that we have concrete homological tools to compute the K-theory. However, one still has to find an explicit projective resolution of the trivial action, preferably of length one. We show that this can be done for several choices of $\G$ provided that we do not consider $\G\wr_{\ast}S_{N}^{+}$ but the monoidally equivalent quantum group $\G\wr_{\ast}SO_{q}(3)$.
\end{itemize}

Let us end this introduction with a short overview of the organization of the paper. Section \ref{sec:preliminaries} contains some background on compact quantum groups and the formalism of based rings. We then classify in Section \ref{sec:torsion} the torsion actions of any free wreath product $\G\wr_{\ast}S_{N}^{+}$. We also investigate the behaviour of torsion actions under free complexification, enven though this is not needed in the sequel. As a consequence, we prove in Theorem \ref{thm:baumconnestorsionfree} that if $\G$ is torsion-free and has the strong Baum-Connes property, then $\G\wr_{\ast}S_{N}^{+}$ also has the strong Baum-Connes property. Eventually, in Section \ref{sec:ktheory} we compute the K-theory of $\G\wr_{\ast}SO_{q}(3)$ for several quantum groups $\G$.

\subsection*{Acknowledgments}

It is a pleasure to thank Roland Vergnioux for stimulating discussions on the content of Section \ref{sec:ktheory}. The first author also thanks Yuki Arano for fruitful conversations on Section \ref{sec:torsion}. Both authors are grateful to the referees for comments and suggestions which improved the quality of the manuscript.

\section{Preliminaries}\label{sec:preliminaries}

In this section we will recall some basic definitions and facts concerning compact quantum groups and torsion, mainly to fix notations. The reader may refer for instance to \cite{neshveyev2014compact} for details and proofs. A \emph{compact quantum group} $\G$ is given by a C*-algebra $C(\G)$ together with a $*$-homomorphism $\D : C(\G) \rightarrow C(\G)\otimes C(\G)$ satisfying the coassociativity condition $(\D\otimes \id)\circ\D = (\id\otimes \D)\circ\D$ and such that both $(1\otimes C(\G))\Delta(C(\G))$ and $(C(\G)\otimes 1)\Delta(C(\G))$ span a dense subspace of $C(\G)\otimes C(\G)$. Our main focus will be representations of such objects. By the results of \cite{woronowicz1987compact}, these representations are always equivalent to a direct sum of finite-dimensional unitary ones, so that we will only define the latter.

\begin{de}
A \emph{unitary representation of dimension $n$} of $\G$ is a unitary element $u\in M_{n}(\C)\otimes C(\G)$ such that
\begin{equation*}
\D(u_{ij}) = \sum_{k=1}^{n}u_{ik}\otimes u_{kj}.
\end{equation*}
A \emph{morphism} between representations $u$ and $v$ of dimension $n$ and $m$ respectively is a linear map $T : \C^{n}\rightarrow \C^{m}$ such that $(T\otimes \id)u = v(T\otimes \id)$. Two representations are said to be \emph{equivalent} if there is a bijective morphism between them. Eventually a representation is said to be \emph{irreducible} if its only self-intertwiners are scalar multiples of the identity.
\end{de}

As an example, there is always a \emph{trivial representation} $\varepsilon_{\G} = 1\otimes 1\in \C\otimes C(\G)$. The equivalence classes of irreducible finite-dimensional unitary representations of $\G$ form a set denoted by $\Irr(\G)$. One way to gather all the information about the representation theory of $\G$ is to consider its \emph{representation category} $\Rep(\G)$ whose objects are all finite-dimensional representations and morphisms are representation morphisms, which is a \emph{rigid tensor C*-category}. If $F\subset \Irr(\G)$, one can consider the smallest full rigid tensor subcategory $\CC_{F}$ of $\Rep(\G)$ containing $F$. By Woronowicz's Tannaka-Krein duality \cite{woronowicz1988tannaka}, there is then an associated C*-subalgebra $C(\HH)\subset C(\G)$ such that restricting the coproduct to $C(\HH)$ endows it with the structure of a compact quantum group $\HH$. Moreover, $\Rep(\HH)$ naturally identifies with $\CC_{F}$ so that we will say that $\widehat{\HH}$ is the quantum discrete subgroup of $\widehat{\G}$ \emph{generated} by $F$.

Defining torsion (or torsion-freeness) for discrete quantum groups is not straightforward. The definition we use emerged from the developement of an analogue of the Baum-Connes conjecture for quantum groups. This conjecture concerns the equivariant KK-theory of C*-algebras acted upon by $\G$, so that we need a dynamical characterization of torsion, which we recall from \cite{meyer2008homological} (see also \cite{voigt2015structure}).

\begin{de}
A \emph{torsion action} $(A, \alpha)$ of a compact quantum group $\G$ is given by a finite-dimensional C*-algebra $A$ together with an action $\alpha : A\rightarrow A\otimes C(\G)$ which is ergodic, i.e. $\{a\in A, \alpha(a) = a\otimes 1\}=\C.1$.
\end{de}

As an example, if $\gamma\in \Irr(\G)$ and $u^{\gamma}\in B(H_{\gamma})\otimes C(\G)$ is a representative, we can define a torsion action $\alpha_{\gamma}$ on $B(H_{\gamma})$ by setting $\alpha_{\gamma}(T) = (u^{\gamma})^{*}(T\otimes 1)u^{\gamma}$. If $\G$ is the dual of a discrete group $\Gamma$, then all irreducible representations are one-dimensional and the above action is always trivial. For a general quantum group, it is trivial up to equivariant Morita equivalence.

\begin{de}
Let $\G$ be a compact quantum group. A torsion action of $\G$ is said to be \emph{trivial} if it is equivariantly Morita equivalent to the trivial action. If all torsion actions are trivial, then $\G$ is said to be \emph{torsion-free}.
\end{de}

\begin{rem}
If $\G$ is the dual of a discrete group $\Gamma$, then it is torsion-free in the above sense if and only if $\Gamma$ is torsion-free. This comes from the fact that finite-dimensional ergodic coactions of $\Gamma$ always arise from finite subgroups of $\Gamma$, see for instance \cite[Prop 4.2]{voigt2015structure}.
\end{rem}

Classifying torsion actions is not an easy task in general. In \cite{arano2015torsion}, Y. Arano and K. De Commer developed a method to deal with this problem in a combinatorial way, through the notion of \emph{fusion ring}. Since this will be fundamental in the present work, we recall the definitions in detail. The reader may find a complete exposition of the subject in \cite[Chap 3]{etingof2015tensor}, where these objects are called \emph{$\Z_{+}$-rings} but we will rather use the conventions and terminology of \cite{arano2015torsion}. Let $I$ be a set with a distinguished element $\varepsilon$ and an involution $i\mapsto \overline{i}$ fixing $\varepsilon$. A ring structure $\otimes$ on $\Z_{I}$ is given by structure constants $\lambda_{i_{1}, i_{2}}^{i_{3}}$ such that
\begin{equation*}
i_{1}\otimes i_{2} = \sum_{i_{3}\in I}\lambda_{i_{1}, i_{2}}^{i_{3}}i_{3},
\end{equation*}
where all but finitely many terms vanish. Let us write $i_{3}\subset i_{1}\otimes i_{2}$ if $\lambda_{i_{1}, i_{2}}^{i_{3}}\neq 0$.

\begin{de}
A $I$-based ring is a ring structure $\otimes$ on $\Z_{I}$ with positive integer structure constants such that
\begin{itemize}
\item $\overline{i_{1}\otimes i_{2}} = \overline{i_{2}}\otimes \overline{i_{1}}$,
\item $\varepsilon\subset \overline{i_{1}}\otimes i_{2}$ if and only if $i_{1} = i_{2}$.
\end{itemize}
A \emph{dimension function} on a based ring is a unital ring homomorphism $d : (\Z_{I}, \otimes)\rightarrow \R$ such that
\begin{itemize}
\item $d(i) > 0$ for all $i\in I$,
\item $d(\overline{i}) = d(i)$.
\end{itemize}
A based ring endowed with a dimension function is called a \emph{fusion ring}.
\end{de}

Let $\G$ be a compact quantum group, set $I = \Irr(\G)$ and endow it with the involution coming from the conjugation of representations. The tensor product of representations turns $\Z_{I}$ into a based ring and it is a fusion ring when endowed with the dimension function $d$ of representations. This ring will be denoted by $R_{\G}$ in the sequel and it is the only type of fusion ring we will be interested in. The main idea of \cite{arano2015torsion} is that torsion actions give rise to specific modules over the fusion ring of a compact quantum group. As before, if $J$ is a set, then any $\Z_{I}$-module structure on $\Z_{J}$ is given by structure constants $\lambda_{i, j_{1}}^{j_{2}}$ such that
\begin{equation*}
i\otimes j_{1} = \sum_{j_{2}\in J}\lambda_{i, j_{1}}^{j_{2}}j_{2},
\end{equation*}
where again all but finitely many terms vanish. We will write $j_{2}\subset i\otimes j_{1}$ if $\lambda_{ij_{1}}^{j_{2}}\neq 0$.

\begin{de}
A \emph{$J$-based module} is a $(\Z_{I}, \otimes)$-module structure on $\Z_{J}$ with positive integer structure constants such that $j_{1}\subset i\otimes j_{2}$ if and only if $j_{2}\subset \overline{i}\otimes j_{1}$. It is said to be \emph{cofinite} if for all $j, j'\in J$, $\{i\in I, j'\subset i\otimes j\}$ is finite. It is said to be \emph{connected} if for any $j, j'\in J$, there exists $i\in I$ such that $j'\subset i\otimes j$. A \emph{compatible dimension function} on a based module is a linear map $d : (\Z_{J}, \otimes) \rightarrow \R$ such that
\begin{itemize}
\item $d(j) > 0$ for all $j\in J$,
\item $d(i\otimes j) = d(i)d(j)$ for all $i\in I$ and $j\in J$.
\end{itemize}
A based module endowed with a compatible dimension function is called a \emph{fusion module}.
\end{de}

Note that if $M$ is a cofinite based module over a fusion ring $R$ with dimension function $d$, then we can define a $\Z_{I}$-valued bilinear form by
\begin{equation*}
\langle j_{1}, j_{2}\rangle = \sum_{i\in I}\lambda_{\overline{i}j_{1}}^{j_{2}}i.
\end{equation*}
which is equivariant in the sense that $\langle i\otimes j_{1}, j_{2}\rangle = i\otimes \langle j_{1}, j_{2}\rangle$. Then, for any $j_{0}\in J$ the map $d : j\mapsto d(\langle j, j_{0}\rangle)$ is a compatible dimension function on $M$.

\begin{ex}
Let $R = \Z_{I}$ be a based ring. The regular action turns it into a based $R$-module and the dimension function on the fusion ring yields a compatible dimension function on the module. It follows from the definition of a based ring that this module is cofinite and connected. A fusion $R$-module $M$ is said to be \emph{standard} if it is isomorphic to $R$ with this fusion module structure.
\end{ex}

Let us now explain the idea of \cite{arano2015torsion}. If $(A, \alpha)$ is a torsion action of a compact quantum group $\G$, then the category of $\G$-equivariant Hilbert modules over $A$ is a \emph{module C*-category} over the representation category of $\G$. As a consequence, its Grothendieck group is a based module over $R_{\G}$. The finite-dimensionality of the action implies that the module is cofinite, while the ergodicity yields connectedness. If the action is trivial, then the module category is exactly the representation category of $\G$, hence the corresponding module is $R_{\G}$. This justifies the following definitions :

\begin{de}
Let $R$ be an $I$-based ring. A based module is said to be a \emph{torsion module} if it is cofinite and connected. The based ring $R$ is said to be \emph{torsion-free} if any torsion module is standard.
\end{de}

According to \cite[Thm 2.8]{arano2015torsion}, a compact quantum group $\G$ is torsion-free if $R_{\G}$ is torsion-free. However, the converse may not hold since it is not clear that any torsion module of $R_{\G}$ can be traced back to a torsion action of $\G$. However, we will see that knowing the torsion modules of $R_{\G}$ is often enough to understand the torsion actions of $\G$. Let us conclude by a remark on the link between torsion actions and monoidal equivalence which will be crucial. If $\G$ and $\HH$ are monoidally equivalent compact quantum groups in the sense of \cite{bichon2006ergodic}, then there is a one-to-one correspondence between their actions which was explicitly described in \cite{de2010actions}. Moreover, this correspondence can be lifted to an equivalence of categories which preserves finite-dimensionality and ergodicity as explained in \cite[Sec 8]{voigt2011baum}. Therefore, the equivariant Morita equivalence classes of torsion actions are also in one-to-one correspondence. This means that we only have to find a simple monoidally equivalent model to study torsion in a given compact quantum group.

\section{Classification of torsion actions}\label{sec:torsion}

In this section we will investigate how torsion actions behave under several constructions. The main goal is to classify the torsion actions of free wreath products in order to be able to prove a suitable version of the Baum-Connes conjecture for them. The general strategy is as follows : we first work at the level of fusion rings and fusion modules, where everything can be treated combinatorially. Then, we use these results to deduce the corresponding statements for genuine actions using the setting of module C*-categories and Tannaka-Krein duality for ergodic actions of compact quantum groups developed in \cite{de2013tannaka}.

\subsection{Free product}

Given two compact quantum groups $\G_{1}$ and $\G_{2}$, S. Wang defined in \cite{wang1995free} their free product $\G_{1}\ast\G_{2}$ (here we are slightly abusing notation since this should be understood as the dual of the free product of the dual discrete quantum groups). Since we will not need the explicit definition but only the representation theory, let us recall what irreducible representations of $\G_{1}\ast\G_{2}$ look like. They can be indexed by alternating words in $\Irr(\G_{1})\setminus\{\varepsilon_{1}\}$ and $\Irr(\G_{2})\setminus\{\varepsilon_{2}\}$ with the following fusion rules :
\begin{equation*}
w\beta_{1}\otimes \beta_{2} w' = w\beta_{1}\beta_{2} w'
\end{equation*}
if $\beta_{1}$ and $\beta_{2}$ are not representations of the same factor and
\begin{equation*}
w\beta_{1}\otimes \beta_{2} w' = \sum_{\underset{\gamma\neq\varepsilon}{\gamma\subset\beta_{1}\otimes \beta_{2}}}w\gamma w' \oplus \delta_{\beta_{1}, \overline{\beta}_{2}}(w\otimes w')
\end{equation*}
otherwise. It was proven in \cite[Thm 1.25]{arano2015torsion} that if $\G_{1}$ and $\G_{2}$ are torsion-free, then their free product also is. However if $(A, \alpha)$ is a torsion action for say $\G_{1}$, then it is also a torsion action for $\G_{1}\ast\G_{2}$ through the inclusion
\begin{equation*}
A\otimes C(\G_{1})\subset A\otimes C(\G_{1}\ast\G_{2}).
\end{equation*}
Thus, the best that we can expect in general is that torsion actions of the free product all arise from this construction. This is what we are going to prove now. To do this, let us first introduce some terminology. If $\G$ and $\HH$ are two compact quantum groups with $C(\HH)\subset C(\G)$, then any action $\alpha$ of $\HH$ gives rise to an action of $\G$ called the \emph{induced action} and denoted by $\Ind_{\HH}^{\G}(\alpha)$. At the level of fusion modules, this corresponds to the usual construction of an induced module, namely considering the $R_{\G}$-module $R_{\G}\otimes_{R_{\HH}}N$, which will be denoted by $\Ind_{R_{\HH}}^{R_{\G}}(N)$. For convenience, let us give a more explicit description in the case of free products. For two base sets $I_{1}$ and $I_{2}$, we will denote by $W(I_{1}, I_{2})$ the set of alternating words ending in $I_{2}$.

\begin{lem}\label{lem:inducedfreeproduct}
Let $R_{1}$ and $R_{2}$ be two based rings with respective base sets $I_{1}$ and $I_{2}$ and let $N$ be a based $R_{1}$-module with base set $J$. The induced module $\Ind_{R_{1}}^{R_{1}\ast R_{2}}(N)$ is the based module with base set
\begin{equation*}
\{wj, j\in J, w\in W(I_{1}, I_{2})\cup\{\emptyset\}\}
\end{equation*}
with the obvious action of $R_{1}\ast R_{2}$.
\end{lem}

\begin{proof}
This directly follows from the description of the representation theory.
\end{proof}

Note that inducing the standard module of one of the factors yields the standard module of the free product. As a first step, we will prove that if we induce a non-standard torsion module, then we again obtain a non-standard one. This requires a characterization of standard module through the following notion of stabilizer : if $R$ is an $I$-based ring and $M$ is a $J$-based $R$-module, we define the \emph{stabilizer} of an element $j\in J$ by
\begin{equation*}
\Stab(j) = \{i\in I, j\subset i\otimes j\}.
\end{equation*}
Note that if $M$ is cofinite, then the stabilizer of any element is finite
. Moreover, if the module is standard then there is at least one element with trivial stabilizer, which is the one corresponding the trivial representation of $\G$. In fact, the converse also holds :

\begin{lem}\label{lem:regularmodules}
Let $N$ be a fusion module of $R_{\G}$ and assume that there is a basis element $j_{0}$ with trivial stabilizer. Then, there is an isomorphism of based $R_{\G}$-modules $N\to R_{\G}$ sending $j_{0}$ to the trivial representation.
\end{lem}

\begin{proof}
Recall from Section \ref{sec:preliminaries} the definition of the bilinear form $\langle j, j'\rangle = \sum_{\beta\in \Irr(\G)}\lambda_{\beta j}^{j'}\overline{\beta}$. By definition, if $j_{0}$ has trivial stabilizer then $\langle j_{0}, j_{0}\rangle = \varepsilon$. Now let $\beta\in \Irr(\G)$ be non-trivial and set
\begin{equation*}
\beta\otimes j_{0} = \sum_{k=1}^{n} \lambda_{k}j_{k}.
\end{equation*}
We then get
\begin{equation*}
\sum_{k}\lambda_{k}\langle j_{k}, j_{0}\rangle = \langle \beta\otimes j_{0}, j_{0}\rangle = \beta\otimes \langle j_{0}, j_{0}\rangle = \beta\otimes \varepsilon = \beta.
\end{equation*}
However, $\langle j_{k}, j_{0}\rangle$ always contains $\beta$ and $\lambda_{k}$ is a non-negative integer. The only possibility is to have $k=1$. In other words, $\beta\otimes j_{0} = j_{1}$ is a basis element. It is now clear that sending $j_{0}$ to the trivial representation gives an isomorphism with the standard module.
\end{proof}

This will be useful in the proofs of the main results of this section. We will also need to know how the induced module splits back over $R_{\G_{1}}$.

\begin{lem}\label{lem:induceddecomposition}
Let $N$ be a torsion module for $R_{\G_{1}}$. Then, $M = \Ind_{\G_{1}}^{\G_{1}\ast \G_{2}}(N)$ splits as a $R_{\G_{1}}$-module into a direct sum of $N$ and standard modules. In particular, $\Ind_{\G_{1}}^{\G_{1}\ast \G_{2}}(N) \simeq \Ind_{\G_{1}}^{\G_{1}\ast \G_{2}}(N')$ as $R_{\G_{1}\ast\G_{2}}$ modules if and only if $N\simeq N'$ as $R_{\G_{1}}$-modules.
\end{lem}

\begin{proof}
Let us consider an arbitrary $R_{\G_{1}}$-submodule $P$ of $M$ and let $w$ be a word of minimal length such that there exists a basis element $j$ of $N$ satisfying $wj\in P$. Assume that $w\neq \emptyset$ and write $w=\gamma_{k}\dots\gamma_{1}$. We cannot have $\gamma_{k}\in\Irr(\G_{1})$ since otherwise $\gamma_{k-1}\cdots\gamma_{1}j\subset \overline{\gamma}_{k}\otimes wj\in P$, contradicting the minimality of $w$. Thus, $\gamma_{k}\in \Irr(\G_{2})$ and the description of the fusion rules for free products implies that $P$ is standard for $R_{\G_{1}}$. Thus, $P$ can only be non-standard if $w$ is the empty word, i.e. if $P$ contains a basis element of $N$, which of course implies $P=N$.

Now if $\Phi : \Ind_{\G_{1}}^{\G_{1}\ast \G_{2}}(N) \rightarrow \Ind_{\G_{1}}^{\G_{1}\ast \G_{2}}(N')$ is an isomorphism of $R_{\G_{1}\ast\G_{2}}$-modules, it can also be seen as an isomorphism of $R_{\G_{1}}$-modules. This isomorphism preserves standard modules so it must also send $N$ isomorphically to $N'$, hence the result.
\end{proof}

For clarity, the proof that torsion actions of a free product are always induced from one of the factors will be split into two parts. First we will work at the level of fusion modules and then recast the proof in the setting of module C*-categories. The argument is inspired by that of \cite[Thm 1.25]{arano2015torsion}
.

\begin{prop}\label{prop:torsionfreeproduct}
Let $M$ be a torsion $R_{\G_{1}\ast\G_{2}}$-module. Then, it is induced from a torsion module of one of the factors.
\end{prop}

\begin{proof}
Let $e$ be a basis element of $M$ and let $N_{1}^{e}$ (resp. $N_{2}^{e}$) denote the sub-$R_{\G_{1}}$-module (resp. the sub-$R_{\G_{2}}$-module) of $M$ generated by $e$. The proof of \cite[Thm 1.25]{arano2015torsion} shows that if both $N_{1}^{e}$ and $N_{2}^{e}$ are standard for all $e$, then $M$ is standard. In particular, it is induced from the standard module of either $R_{\G_{1}}$ or $R_{\G_{2}}$. 

Let us assume that $M$ is not standard and let $e$ be such that $N_{1}^{e}$ is not standard (the case where $N_{2}^{e}$ is not standard is similar). There is a natural candidate for the isomorphism : for a word $w\in \Irr(\G_{1}\ast\G_{2})$ ending in $\Irr(\G_{2})$ and a basis element $j$ of $N_{1}^{e}$, it should send the basis element $wj$ of $\Ind_{\G_{1}}^{\G_{1}\ast\G_{2}}(N_{1}^{e})$ to $w\otimes j$. However, this only makes sense if $w\otimes j$ is also a basis element. Let us prove this by induction on the length of $w$, with the following induction hypothesis :
\begin{center}
$H_{k}$ : "Let $w = \gamma_{k}\cdots\gamma_{1}\in \Irr(\G_{1}\ast\G_{2})$ be a word ending in $\Irr(\G_{2})$ and let $j$ be a basis element of $N_{1}^{e}$. Then, $w\otimes j$ is a basis element."
\end{center}

For $k=1$, assume that there exists $\beta_{2}\in \Irr(\G_{2})\setminus\{\varepsilon_{2}\}$ such that $j\subset \beta_{2}\otimes j$ and let $\beta_{1}$ be a nontrivial element in $\Stab(j)\cap\Irr(\G_{1})$, which exists by Lemma \ref{lem:regularmodules} because $N_{1}^{e}$ is not standard. Then, for any integer $l$, $(\beta_{1}\beta_{2})^{l}\in \Irr(\G_{1}\ast\G_{2})$ is a non-trivial stabilizer of $j$. This yields infinitely many stabilizers, contradicting cofiniteness. Thus, $\Stab(j)\cap\Irr(\G_{2}) = \{\varepsilon_{2}\}$. This implies by assumption that the $R_{\G_{2}}$-submodule generated by $j$ is standard and the isomorphism must send $j$ to $\varepsilon_{2}$. In particular, $\gamma_{1}\otimes j$ is a basis element for all $\gamma_{1}\in \Irr(\G_{2})$ and $H_{1}$ holds.

Assume now $H_{k}$ for some $k\geqslant 1$ and let $\gamma_{k}\cdots\gamma_{1}\in \Irr(\G_{1}\ast\G_{2})$ be a word ending in $\Irr(\G_{2})$. Let us assume for simplicity that $\gamma_{k}\in \Irr(\G_{2})\setminus\{\varepsilon_{2}\}$, the other case being similar. Set $j' = \gamma_{k}\cdots\gamma_{1}\otimes j$ which is a basis element by $H_{k}$. By the same argument as before, $\Stab(j')\cap\Irr(\G_{1}) = \{\varepsilon_{1}\}$ so that $N_{1}^{j'}$ is standard for $R_{\G_{1}}$ with an isomorphism sending $j'$ to the trivial representation. In particular, for any $\gamma_{k+1}\in \Irr(\G_{1})\setminus\{\varepsilon_{1}\}$, $\gamma_{k+1}\otimes j'$ is a basis element and $H_{k+1}$ holds.

We can now finish the proof. Let
\begin{equation*}
\Phi : \Ind_{\G_{1}}^{\G_{1}\ast\G_{2}}(N_{1}^{e}) \rightarrow M
\end{equation*}
be the map sending $wj$ to $w\otimes j$. This is a surjective module homomorphism by connectedness. Let $w\neq w'$ be words ending in $\Irr(\G_{2})$ and let $j, j'$ be basis elements in $N_{1}^{e}$ such that $\Phi(wj) = \Phi(w'j')$. Then,
\begin{equation*}
j'\subset(\overline{w}'\otimes w)\otimes j.
\end{equation*}
Observe that $\overline{w}'\otimes w$ is a sum of non-empty words starting and ending in $\Irr(\G_{2})\setminus\{\varepsilon_{2}\}$. In particular, there exists $w''\in \Irr(\G_{1}\ast\G_{2})$ ending in $\Irr(\G_{2})$ such that $j'\subset w''\otimes j$. Since we have proved that $w''\otimes j$ is a basis element, we get $j' = w''\otimes j$. Now, if $\beta, \beta'\in \Irr(\G_{1})\setminus\{\varepsilon_{1}\}$ stabilize respectively $j$ and $j'$, we get for any $l\in \N$ a stabilizer $(\overline{w}''\beta'w''\beta)^{l}$ of $j$, contradicting cofiniteness. Thus, $w=w'$, $j=j'$ and $\Phi$ is faithful.
\end{proof}

We can now state and prove the main result of this subsection.

\begin{thm}\label{thm:actionsfreeproduct}
Let $\G_{1}$ and $\G_{2}$ be two compact quantum groups. Then, up to equivariant Morita equivalence there is a one-to-one correspondance between torsion actions of $\G_{1}\ast\G_{2}$ and torsion actions induced from $\G_{1}$ or $\G_{2}$.
\end{thm}

\begin{proof}
As in \cite[Thm 3.16]{arano2015torsion}, it suffices to recast the previous proof in the setting of module C*-categories. Let $(A, \alpha)$ be a torsion action of $\G_{1}\ast\G_{2}$ and consider the associated C*-module category $(\mathcal{C}, \otimes)$. For any irreducible object $X$ in $\mathcal{C}$, we can consider the module C*-categories $\mathcal{C}_{1}^{X}$ and $\mathcal{C}_{2}^{X}$ generated by $X$ and the action of the representation categories of $\G_{1}$ and $\G_{2}$ respectively. By Proposition \ref{prop:torsionfreeproduct}, there is an $X$ such that one of them (say $\mathcal{C}_{1}^{X}$) has a non-standard fusion module if $(A, \alpha)$ is not trivial. By \cite[Lem 3.10]{arano2015torsion}, the module C*-category is equivalent to the one of the trivial action if its associated fusion module is standard. The same reasoning as for fusion modules therefore yields an isomorphism between $\mathcal{C}$ and the module C*-category induced from $\mathcal{C}_{1}^{X}$ (by which we mean the module C*-category obtained by induction of all the objects of $\mathcal{C}_{1}^{X}$). To conclude, note that by the general results of \cite{de2013tannaka}, there is a torsion action $(A', \alpha')$ of $\G_{1}$ such that the associated module C*-category is $\mathcal{C}_{1}^{X}$. Thus, $\mathcal{C}$ is equivalent to the module C*-category associated to $\Ind_{\G_{1}}^{\G_{1}\ast\G_{2}}(\alpha')$ and again by \cite{de2013tannaka} the actions are equivariantly Morita equivalent.

Consider now two induced torsion actions which are equivariantly Morita equivalent. If they are induced from different factors, then the fusion module of the one induced from $\G_{2}$ is a direct sum of standard modules when restricted to $\G_{1}$. Thus, the fusion module associated to the one induced from $\G_{1}$ is isomorphic to a direct sum of standard module for $R_{\G_{1}}$. Since all its submodules are also standard for $R_{\G_{2}}$, we conclude by \cite[Thm 1.25]{arano2015torsion} that the actions are trivial. This leaves us with the case of two torsion actions $(A, \alpha)$ and $(A', \alpha')$ of say $\G_{1}$ such that $\Ind_{\G_{1}}^{\G_{1}\ast\G_{2}}(\alpha)$ and $\Ind_{\G_{1}}^{\G_{1}\ast\G_{2}}(\alpha')$ are equivariantly Morita equivalent. The same reasoning as in Lemma \ref{lem:induceddecomposition} shows that in both associated module C*-categories $\mathcal{C}$ and $\mathcal{C}'$, the module C*-subcategory coming from the original action is the only one to be non-trivial over the representation category of $\G_{1}$. The equivalence of categories must therefore restrict to an equivalence between these sub-categories and we conclude by \cite{de2013tannaka}.
\end{proof}

\subsection{Free wreath product}\label{subsec:wreathproduct}

In \cite{bichon2004free}, J. Bichon introduced a "free" version of the wreath product construction which produces, out of a compact quantum group $\G$ and a quantum permutation group $S_{N}^{+}$, their \emph{free wreath product} $\G\wr_{\ast}S_{N}^{+}$. This gives many examples of compact quantum groups and in particular the so-called \emph{quantum reflection groups} when $\G$ is the dual of a cyclic group. Once again, we will not give the definition of these objects since it is not necessary for our purpose. All we need is to understand their fusion ring, which was described by F. Lemeux and P. Tarrago in \cite{lemeux2014free}. Consider the free monoid $F$ over $\Irr(\G)$. Then, $\Z_{F}$ is an abelian group. Moreover, we can define an involution on it by setting
\begin{equation*}
\overline{\gamma_{1}\cdots\gamma_{k}} = \overline{\gamma}_{k}\cdots\overline{\gamma}_{1}.
\end{equation*}
We can also use the fusion rules of $\G$ to define a ring structure on $\Z_{F}$ by
\begin{equation*}
(\gamma_{1}\cdots\gamma_{k})\otimes (\gamma'_{1}\cdots\gamma'_{n}) = \gamma_{1}\cdots\gamma_{k}\gamma'_{1}\cdots\gamma'_{n}\oplus \left(\sum_{\beta\subset\gamma_{k}\otimes\gamma_{1}'}\gamma_{1}\cdots\gamma_{k-1}\beta\gamma'_{2}\cdots\gamma'_{n}\right) \oplus \delta_{\gamma_{k}, \overline{\gamma}'_{1}}(\gamma_{1}\cdots\gamma_{k-1})\otimes (\gamma'_{2}\cdots\gamma'_{n}).
\end{equation*}
If $N\geqslant 4$, the based ring obtained in this way is the fusion ring of $\G\wr_{\ast}S_{N}^{+}$.

The idea to study torsion in $\G\wr_{\ast}S_{N}^{+}$ is to embed it into a compact quantum group whose torsion is better understood, namely a free product. Such an embedding need not exist in general, but it always does once we consider a monoidally equivalent compact quantum group, thanks to the following result of F. Lemeux and P. Tarrago in \cite{lemeux2014free}. From now on, let us denote by $u^{1}$ the fundamental representation of $SU_{q}(2)$.

\begin{thm}[Lemeux-Tarrago]\label{thm:lemeuxtarrago}
Let $\G$ be a compact quantum group and let $\HH_{q}$ be the compact quantum subgroup of $\G\ast SU_{q}(2)$ generated by $\{u^{1}\alpha u^{1}\mid \alpha\in \Irr(\G)\}$. If $N\geqslant 4$, then there exists $0< \vert q\vert < 1$ such that $\HH_{q}$ is monoidally equivalent to $\G\wr_{\ast}S_{N}^{+}$.
\end{thm}

\begin{rem}
P. Fima and L. Pittau introduced in \cite{fima2015free} the free wreath product of a compact quantum group by any quantum automorphism group of a finite-dimensional C*-algebra $B$ preserving a given state and proved a similar monoidal equivalence. In fact, $\HH_{q}$ is nothing but $\G\wr_{\ast} SO_{q}(3)$ in the sense of \cite[Def 2.6]{fima2015free}.
\end{rem}

We will therefore classify the torsion actions of $\HH_{q}$. The point is that any torsion action of $\HH_{q}$ induces a torsion action of $\G\ast SU_{q}(2)$. Since $SU_{q}(2)$ is known to be torsion-free by \cite[Prop 3.2]{voigt2011baum} (see also \cite[Prop 1.23]{arano2015torsion} for a combinatorial proof), the torsion actions of the free product are induced from torsion actions of $\G$ by Theorem \ref{thm:actionsfreeproduct}. For convenience, let us give an explicit description of the inclusion $\Lambda : R_{\HH_{q}}\subset R_{\G\ast SU_{q}(2)}$. Let us denote by $u^{n}$ the $n$-th irreducible representation of $SU_{q}(2)$ and let $w$ be a word in the free monoid over $\Irr(\G)$.
\begin{itemize}
\item If $w = \varepsilon_{\G}^{n}$, then $\Lambda(w) = u^{2n}$,
\item if $w = \beta_{1}\cdots \beta_{n}$ with $\beta_{i}\neq \varepsilon_{\G}$ for all $1\leqslant i\leqslant n$, then $\Lambda(w) = u^{1}\beta_{1}u^{2}\beta_{2}u^{2}\cdots u^{2}\beta_{n}u^{1}$,
\item otherwise, we can write $w$ as $w = \varepsilon^{n_{1}}_{\G}w_{1}\varepsilon^{n_{2}}_{\G}w_{2}\cdots w_{k}\varepsilon^{n_{k+1}}_{\G}$ where each $w_{i}$ does not contain $\varepsilon_{\G}$. Let us denote by $\widetilde{\Lambda}(w_{i})$ the same expression as $\Lambda(w_{i})$ except that the first and last $u^{1}$'s are removed. Then, $\Lambda(w) = u^{2n_{1}+1}\widetilde{\Lambda}(w_{1})u^{2n_{2}+2}\widetilde{\Lambda}(w_{2})\cdots u^{2n_{k}+2}\widetilde{\Lambda}(w_{k})u^{2n_{k+1}+1}$.
\end{itemize}
One can easily check that these expressions are compatible with the tensor products and that $\Lambda$ is an isomorphism. What we now have to understand is how torsion modules of $\G\ast SU_{q}(2)$ behave when we only consider the action of $\HH_{q}$.

\begin{lem}\label{lem:torsionwreathproduct}
Let $N$ be a torsion module for $R_{\G}$ and let $M = \Ind_{\G}^{\G\ast SU_{q}(2)}(N)$. Then, $M$ contains a unique non-standard torsion module for $R_{\HH_{q}}$.
\end{lem}

\begin{proof}
Let $j$ be a basis element of $N$, let $w\in \Irr(\G\ast SU_{q}(2))$ be a word ending in $\Irr(SU_{q}(2))$ and let $N(wj)$ be the sub-$R_{\HH_{q}}$-module generated by the basis element $wj$ of $M$. We will denote by $w'$ a word of minimal length such that $w'j \in N(wj)$. If $w' = \emptyset$, then for any word $w''$ ending in $\Irr(SU_{q}(2))\setminus\{\varepsilon\}$, $\Lambda(w'')\otimes j = \Lambda(w'')j$ so that $N(j)$ is standard. If $w'$ starts in $\Irr(\G)$, then again $\Lambda(w'')\otimes w'j = \Lambda(w'')w'j$ for any word $w''$ ending in $\Irr(SU_{q}(2))\setminus\{\varepsilon\}$, so that $N(wj) = N(w'j)$ is standard. Let us therefore assume that $w' = u^{n_{0}}\beta_{1}\cdots\beta_{k}u^{n_{k}}$. If $n_{0} > 1$, then since $u^{2}\in \Irr(\HH_{q})$ we see that
\begin{equation*}
u^{n_{0}-2}\beta_{1}\cdots\beta_{k}u^{n_{k}}j\in N(wj)
\end{equation*}
and we can assume that $n_{0} = 1$. If $k>0$, then tensoring by $u^{1}\overline{\beta}_{1}u^{1}$ we see that $u^{n_{1}}\beta_{2}\cdots\beta_{k}u^{n_{k}}j\in N(wj)$, contradicting the minimality of $w'$. Thus, $w' = u^{1}$. Note that for any basis element $j'$ of $N$, there exists $\beta\in \Irr(\G)$ such that $j'\subset \beta\otimes j$, therefore $u^{1}j'\subset u^{1}\beta u^{1}\otimes u^{1}j$. This implies that $N(u^{1}j) = N(u^{1}j')$ and we therefore get only one torsion module, denoted by $N_{u^{1}}$.

To conclude, we now have to show that $N_{u^{1}}$ is not standard. But this is clear since $u^{2}$ is a non-trivial stabilizer of all its basis elements.
\end{proof}

As for free products, the statement for actions can be deduced from this algebraic result.

\begin{thm}\label{thm:torsionwreathproduct}
Let $\G$ be a compact quantum group and let $N\geqslant 4$. Then, the equivariant Morita equivalence classes of \emph{non-trivial} torsion actions of $\G\wr_{\ast}S_{N}^{+}$ are in one-to-one correspondence with \emph{all} the equivariant Morita equivalence classes of torsion actions of $\G$.
\end{thm}

\begin{proof}
Let $(A, \alpha)$ be a torsion action of $\HH_{q}$. By Theorem \ref{thm:actionsfreeproduct}, $\Ind_{\HH_{q}}^{\G\ast SU_{q}(2)}(\alpha)$ is equivariantly Morita equivalent to an action induced from $\G$. By Lemma \ref{lem:torsionwreathproduct}, the restriction of such an action to $\HH_{q}$ has exactly one non-trivial summand. Thus, this summand is equivariantly Morita equivalent to $(A, \alpha)$. Moreover, if two torsion actions of $\HH_{q}$ are equivariantly Morita equivalent, then the same holds for their induction to $\G\ast SU_{q}(2)$, so that by Theorem \ref{thm:actionsfreeproduct} the original actions of $\G$ are also equivariantly Morita equivalent.
\end{proof}

In particular, a free wreath product is never torsion-free since the trivial action of $\G$ gives rise to a non-trivial torsion action of $\G\wr_{\ast}S_{N}^{+}$. Let us describe explicitly this action. Consider the quantum subgroup of $\HH_{q}$ generated by $u^{2}$, which is isomorphic to $SO_{q}(3)$. It is known that it admits a projective action $\alpha_{q}$ on $M_{2}(\C)$ which is a non-trivial torsion action. It was proven in \cite[Lem 4.4]{voigt2015structure} that this is the only non-trivial torsion action of $SO_{q}(3)$ up to equivariant Morita equivalence. Note that this action becomes trivial when induced to $SU_{q}(2)$. Similarly, by Theorem \ref{thm:torsionwreathproduct} $\Ind_{SO_{q}(3)}^{\HH_{q}}(\alpha_{q})$ is a nontrivial torsion action whose induction to $\G\ast SU_{q}(2)$ is trivial. Under the monoidal equivalence, $(M_{2}(\C), \alpha_{q})$ becomes the defining action of $C(S_{N}^{+})$ on $\C^{N}$. Inducing this action to $\G\wr_{\ast}S_{N}^{+}$ yields a nontrivial torsion action $(\C^{N}, \alpha_{N})$ which is precisely the one obtained from the trivial action of $\G$. If $\G$ is torsion-free, this is the only source of torsion in the free wreath product :

\begin{cor}
Let $\G$ be a torsion-free compact quantum group and let $N\geqslant 4$. Then, $(\C^{N}, \alpha_{N})$ is the only non-trivial torsion action of $\G\wr_{\ast}S_{N}^{+}$ up to equivariant Morita equivalence.
\end{cor}

\subsection{Free complexification}

Before ending this section, we want to investigate one last construction called \emph{free complexification} and introduced by T. Banica in \cite{banica2007note}. This is an important source of free unitary quantum groups and may therefore be an interesting object from the point of view of K-theory computations. However, it turns out that in that case, understanding torsion is more difficult than for free wreath products. We will give some general results and then work out the particular case of the complexified hyperoctahedral quantum group $\widetilde{H}_{N}^{+}$. We start by recalling the definition of free complexification, which is slightly different from the previous ones because it does not only take a compact quantum group as an argument but also a distinguished representation. A representation $u$ of a compact quantum group $\G$ is said to be a \emph{fundamental representation} of $\G$ if any irreducible representation is contained up to equivalence in a tensor power of $u$ and its conjugate.

\begin{de}
A \emph{compact matrix quantum group} is a pair $(\G, u)$ where $\G$ is a compact quantum group and $u$ is a fundamental representation of $\G$. If moreover $\overline{u} = u$, then $(\G, u)$ is said to be \emph{orthogonal}.
\end{de}

Let $S^{1}$ denote the compact group of complex numbers of modulus $1$ and consider the free product $\G\ast S^{1}$. If $z$ denotes the identity representation of $S^{1}$, then $\widetilde{u} = uz$ is a representation of $\G\ast S^{1}$ and we can consider the quantum subgroup $\widetilde{\G}$ generated by this representation. This yields a new compact matrix quantum group $(\widetilde{\G}, \widetilde{u})$ called the \emph{free complexification} of $(\G, u)$.

Before studying the behaviour of torsion actions under this construction, let us prove a result concerning divisibility. The notion of divisible quantum subgroup was introduced in \cite{vergnioux2013k} to study the Baum-Connes conjecture for the free unitary quantum groups $U_{F}^{+}$. There are several equivalent definitions but we only give the one which will be useful for us. Let $\G$ and $\HH$ be compact quantum groups such that $C(\HH)\subset C(\G)$. Then, $\Irr(\HH)$ embeds into $\Irr(\G)$ and we can define an equivalence relation on $\Irr(\G)$ by setting $\beta\sim\beta'$ if there exists $\gamma\in \Irr(\HH)$ such that $\beta'\subset \beta\otimes \gamma$.

\begin{de}
The quantum group $\HH$ is said to be \emph{divisible} in $\G$ if for any class $X\in \Irr(\G)/\Irr(\HH)$, there exists a representative $\beta$ of $X$ such that for all $\gamma\in \Irr(\HH)$, $\beta\otimes \gamma$ is irreducible.
\end{de}

It was proven in \cite[Prop 4.3]{vergnioux2013k} that $U_{F}^{+}$ is divisible in $O_{F}^{+}\ast S^{1}$, where $O_{F}^{+}$ denotes the corresponding free orthogonal quantum group. We will extend this result to arbitrary compact quantum groups provided that the fundamental representation is orthogonal. To do so, we first introduce an important subgroup of $(\G, u)$.

\begin{de}
Let $(\G, u)$ be an orthogonal compact matrix quantum group. The compact quantum subgroup $\G_{ev}$ of $\G$ generated by $u\otimes u$ is called the \emph{even part} of $(\G, u)$.
\end{de}

Equivalently, the irreducible representations of $\G_{ev}$ are exactly the irreducible representations of $\G$ which are subrepresentations of $u^{\otimes 2k}$ for some $k\in\N$. Note that because $u\otimes u\subset (uz)\otimes (\overline{z}u)$, $C(\G_{ev})\subset C(\widetilde{\G})$.

\begin{prop}\label{prop:divisiblefreecomplexification}
Let $(\G, u)$ be an orthogonal compact matrix quantum group. Then, the free complexification $\widetilde{\G}$ is divisible in $\G\ast S^{1}$.
\end{prop}

\begin{proof}
If $C(\G_{ev}) = C(\G)$, then $C(\G)\subset C(\widetilde{\G})$. This implies that $z\subset u\otimes (uz)\in \Irr(C(\widetilde{\G}))$ and eventually that $\widetilde{\G} = \G\ast S^{1}$. In that case, the result is trivial.

Let us therefore assume that $\G_{ev}\neq \G$ and describe the irreducible representations of $\widetilde{\G}$ inside $\G\ast S^{1}$. Following the proof of \cite[Prop 4.3]{vergnioux2013k}, we set, for $\epsilon\in \{-1, 1\}$, $[\epsilon]_{-} = \min(\epsilon, 0)$ and $[\epsilon]_{+} = \max(\epsilon, 0)$. Let $W$ be the subset of $\Irr(\G\ast S^{1})$ consisting in words of the form
\begin{equation*}
z^{[\epsilon_{0}]_{-}}\beta_{1}z^{\epsilon_{1}}\cdots z^{\epsilon_{p-1}}\beta_{p}z^{[\epsilon_{p}]_{+}}
\end{equation*}
with $\epsilon_{i+1} = -\epsilon_{i}$ if $\beta_{i}\in\Irr(\G_{ev})$ and $\epsilon_{i+1} = \epsilon_{i}$ if $\beta_{i}\in \Irr(\G)\setminus\Irr(\G_{ev})$. By definition, $W$ is stable under taking tensor products and contragredients. Moreover, it contains $uz$ so that $\Irr(\widetilde{\G})\subset W$. One can prove that the reverse inclusion also holds but we will not need it hereafter.

Let now $X\in \Irr(\G\ast S^{1})/\Irr(\widetilde{\G})$ and let $w$ be a representative of $X$ of minimal length. If $w$ is the trivial representation then we are done. Assume therefore that $w$ has length $1$. If $w=z^{k}$ with $k\in \Z^{*}$, then the result is also clear. If $w=\beta$ for some representation $\beta\in \Irr(\G)$, we have two cases : if $\beta\in \Irr(\G_{ev})$ then tensoring with $\overline{\beta}$ yields $\varepsilon\in X$. Otherwise, tensoring by $\overline{\beta}z$ yields $z\in X$ and we can again conclude. Let us eventually assume that $w$ has length at least $2$. If it ends with $z^{k}$ for some $k\neq 1$, the result is clear. Otherwise, we can as before reduce the length of $w$, hence the result.
\end{proof}

A consequence of Proposition \ref{prop:divisiblefreecomplexification} is that if $\G$ is torsion free and has the strong Baum-Connes property, then so does $\widetilde{\G}$. This follows from \cite[Prop 1.28]{arano2015torsion} and \cite[Thm 6.6]{vergnioux2013k} but does not settle the general case. We therefore have to investigate how torsion may pass to the free complexification. Surprisingly, the answer is not as clear as before because it involves the relationship between modules on $R_{\G}$ and modules on $R_{\G_{ev}}$.

Let us fix an orthogonal compact matrix quantum group $(\G, u)$ once and for all. For convenience, we will say that an irreducible representation $\beta\in \Irr(\G)$ is \emph{even} if it is in $\Irr(\G_{ev})$ and \emph{odd} otherwise. Of course, everything is trivial if $\G_{ev} = \G$. We will therefore assume from now on that this is not the case, i.e. that there exist odd representations. Given a torsion module $N$ on $R_{\G}$, we want to understand how it splits over $\G_{ev}$ and in particular how many connected components there will be. This can be captured by the following equivalence relation on the basis elements $J$ of $N$ : $j\sim j'$ if there exists $\beta\in \Irr(\G_{ev})$ such that $j\subset \beta\otimes j'$. This is an equivalence relation since it means that the two basis elements generate the same $R_{\G_{ev}}$-submodule. Intuitively, $\G_{ev}$ is an "index two" subgroup of $\G$ and we may therefore expect that the number of equivalence classes is at most two.

\begin{prop}\label{prop:torsionevenpart}
There are at most two equivalence classes. Moreover, the following are equivalent :
\begin{enumerate}
\item All basis elements have an odd stabilizer.
\item There exists a basis element with an odd stabilizer.
\item There is only one equivalence class, i.e. the module $N$ is $R_{\G_{ev}}$-connected.
\end{enumerate}
\end{prop}

\begin{proof}
Let $j_{1}, j_{2}, j_{3}$ be basis elements such that the first two are not equivalent to $j_{3}$. By connectedness, there exists $\beta_{1}, \beta_{2}\in \Irr(\G)$ such that $j_{1}\subset \beta_{1}\otimes j_{3}$ and $j_{3}\subset\beta_{2}\otimes j_{2}$. Moreover, $\beta_{1}$ and $\beta_{2}$ are odd by assumption, so that any subrepresentation of $\beta_{1}\otimes\beta_{2}$ is even. Since $j_{1}\subset (\beta_{1}\otimes \beta_{2})\otimes j_{3}$, we get $j_{1}\sim j_{3}$.

As for the equivalence of the statements, we will prove that each one implies the next one. The first implication is trivial. For the second one, let $j$ be a basis element with an odd stabilizer $\beta$, let $j'$ be another basis element and let $\gamma\in \Irr(\G)$ satisfy $j\subset\gamma\otimes j'$. If $\gamma$ is even, then we are done. Otherwise, $j\subset (\beta\otimes \gamma)\otimes j'$ and the tensor product contains only even representations, hence the result. For the third implication, let $j$ be a basis element and let $j'\subset u\otimes j$. Then, there exists an even representation $\gamma$ such that $j\subset \gamma\otimes j'$, so that $\gamma\otimes u$ contains an odd stabilizer of $j$.
\end{proof}

There is always at least one module not satisfying the above equivalent conditions, namely the standard module of $R_{\G}$, since $\varepsilon$ then has trivial stabilizer. More precisely, as $R_{\G_{ev}}$-modules we have $R_{\G} = R_{\G_{ev}}\oplus R_{odd}$, where $R_{odd}$ is the span of all odd representations (which is also the submodule generated by $u$). However, there can also be other non-standard $R_{\G}$-modules whithout odd stabilizers. For instance, the action of $R_{\Z_{4}}$ on the two-dimensional module $\Z.j_{0}\oplus\Z.j_{1}$ where even elements stabilize each point while odd elements exchange them.

We can now link this with torsion modules on the free complexification. To do this, first note that any torsion module of $R_{\widetilde{\G}}$ can be induced to a torsion module of $\G\ast S^{1}$ and this induced module is by Theorem \ref{thm:actionsfreeproduct} isomorphic to one induced from $R_{\G}$. Thus, we can focus on modules coming from $\G$.

\begin{prop}\label{prop:torsionfreecomplexification}
Let $N$ be a non-standard torsion module of $R_{\G}$ and let $M = \Ind_{\G}^{\G\ast S^{1}}(N)$. Then, $M$ contains one or two non-standard torsion $R_{\widetilde{\G}}$-submodules, denoted by $P$ and $P'$. Moreover, $P = P'$ if and only if $N$ is $R_{\G_{ev}}$-connected.
\end{prop}

\begin{proof}
Recall that the basis of $M$ consists in elements of the form $wj$ with $w$ a word on $\{z^{k}, k\neq 0\}\cup\Irr(\G)\setminus\{\varepsilon_{\G}\}$ not ending in $\Irr(\G)$. Let $P$ be a $R_{\widetilde{\G}}$-submodule and consider a word $w$ and a basis element $j$ of $N$ such that $wj\in P$ and $w$ has minimal length. We have several cases :
\begin{itemize}
\item If $w$ starts with an element $\gamma$ of $\Irr(\G_{ev})\subset\Irr(\widetilde{\G})$, then letting $\overline{\gamma}$ act we see that we can reduce the length of $w$, contradicting minimality.
\item If $w$ starts with $z^{k}$ with $k\neq -1$, then the description of the inclusion $\Irr(\widetilde{\G})\subset \Irr(\G\ast S^{1})$ above implies that corresponding module is standard.
\item If $w$ starts with an element $\beta$ of $\Irr(\G)\setminus\Irr(\G_{ev})$ and is of the form $\beta z^{k}w'$, then tensoring with $z^{-1}\overline{\beta}$ shows that we can replace $w$ by $z^{k-1}w'$, contradicting minimality.
\item If $w = \beta\in\Irr(\G)\setminus\Irr(\G_{ev})$, tensoring by $z^{-1}\overline{\beta}$ shows that we can assume $w=z^{-1}$. Let then $j'\neq j$ be such that there is $\gamma\in \Irr(\G)\setminus\Irr(\G_{ev})$ such that $j'\subset \gamma\otimes j$. Then $j'\subset \gamma z\otimes z^{-1}j$, contradicting minimality.
\end{itemize}
As a conclusion, the word $w$ must be empty and we have to consider submodules generated by basis elements of $N$. Let $j$ be such a basis element and let $P_{j}$ be the $R_{\widetilde{\G}}$-submodule generated by $j$. Because $R_{\G_{ev}}\subset R_{\widetilde{\G}}$, if $j\sim j'$ then $P_{j} = P_{j'}$. Reciprocally, any word $w\in \Irr(\widetilde{\G})$ of length at most two contains a power of $z$, hence cannot connect two basis elements of $N$. Thus, by Proposition \ref{prop:torsionevenpart}, all the sub-modules $P_{j}$ coincide if and only if $N$ is connected as a $R_{\G_{ev}}$-module. Otherwise, we get two distinct submodules $P$ and $P'$.

Assume that $P$ and $P'$ are both standard. Then, $M$ was the standard module of $R_{\G\ast S^{1}}$, hence by Theorem \ref{thm:actionsfreeproduct} $N$ was standard, contradicting the assumption. As a consequence, there is at least one non-standard torsion $R_{\widetilde{\G}}$-submodule.
\end{proof}

The main problem with the previous statement is that when all the stabilizers are even, we get two non-standard modules so that there may be more torsion in $\widetilde{\G}$ than in $\G$. It is moreover not clear whether the torsion modules obtained are isomorphic to one another, or isomorphic to those arising from other torsion modules. At least, in the particular case of $\widetilde{H}_{N}^{+}$ we can settle this issue. Recall that $H_{N}^{+} = \Z_{2}\wr_{\ast}S_{N}^{+}$ is the free hyperoctahedral quantum group and $\widetilde{H}_{N}^{+}$ is its free complexification.

\begin{cor}\label{cor:torsioncomplexifiedhyperoctahedral}
The quantum group $\widetilde{H}_{N}^{+}$ has exactly two non-trivial torsion actions up to equivariant Morita equivalence.
\end{cor}

\begin{proof}
By Theorem \ref{thm:torsionwreathproduct}, $H_{N}^{+}$ has two non-trivial torsion actions up to equivariant Morita equivalence. One of them comes from the trivial $R_{\Z_{2}}$-module $\Z.j_{0}$. At the level of the free wreath product, it is generated by $u^{1}j_{0}$ and this element has the odd stabilizer $u^{1}su^{1}$, where $s$ denotes the generator of $\Z_{2}$ seen as an irreducible representation. Thus, $P = P'$ in that case.

As for the second one, it comes from the standard module $\Z.j_{+}\oplus \Z j_{-}$ of $\Z_{2}$ where $s$ exchanges the two basis elements. In that case, all stabilizers are even so that $P \neq P'$. However, there is a $R_{H_{N}^{+}}$-equivariant automorphism $\theta$ of $N$ given by $u^{1}j_{+} \to u^{1}j_{-}$ which restricts to an isomorphism between the two summands of the restriction of $N$ to the even part of $H_{N}^{+}$. Thus, $P$ is isomorphic to $P'$.

As before, the proof can be recast in the context of module categories to yield the result. Moreover, if two torsion actions are equivariantly Morita equivalent, so is their induction to the free product so that by Theorem \ref{thm:actionsfreeproduct} we can conclude that they come from the same torsion action of $H_{N}^{+}$.
\end{proof}

\subsection{Application to the Baum-Connes conjecture}

Classifying the torsion actions of a discrete quantum group is a first step towards the statement of a suitable form of the Baum-Connes conjecture. In this section we will give such a statement and prove it, essentially following \cite{voigt2015structure}. The argument involves duality for quantum groups and therefore requires some conventions. Given a compact quantum group $\G$, we will denote by $\widehat{\G}$ its dual discrete quantum group. The C*-algebra $C^{*}(\widehat{\G})$ is then isomorphic to $C(\G)$ and the Baum-Connes property concerns $\widehat{\G}$. More precisely, it concerns the category $KK^{\widehat{\G}}$ of $\widehat{\G}$-C*-algebras with morphisms given by the equivariant KK-groups. By Baaj-Skandalis duality \cite{baaj1993unitaires}, $A\rightarrow \G\ltimes A$ implements an equivalence of categories $KK^{\G}\rightarrow KK^{\widehat{\G}}$ so that we can chose either the compact or discrete point of view. Let now $\G$ be a compact quantum group and let $C(\HH_{q})\subset C(\G\ast SU_{q}(2))$ be the compact quantum group defined in Theorem \ref{thm:lemeuxtarrago}. Monoidal equivalence implements an equivalence of categories $KK^{\G\wr_{\ast}S_{N}^{+}}\rightarrow KK^{\HH_{q}}$ so that it is enough to study $\HH_{q}$ or its dual.

We now introduce some terminology from \cite{meyer2008homological}. Let $\T_{\HH_{q}}\subset KK^{\HH_{q}}$ be the full subcategory generated by all $\HH_{q}$-C*-algebras of the form $B\otimes C$ where $B$ is a torsion action of $\HH_{q}$ and the action is only on the first tensor factor. It is known by \cite{meyer2006baum} that $KK^{\HH_{q}}$ is a triangulated category so that we may consider the localizing subcategory $\langle \T_{\HH_{q}}\rangle\subset KK^{\HH_{q}}$ generated by $\T_{\HH_{q}}$ and denote by $\langle\CI_{\widehat{\HH}_{q}}\rangle\subset KK^{\widehat{\HH}_{q}}$ the corresponding subcategory. We will say that the discrete quantum group $\widehat{\HH}_{q}$ satisfies the \emph{strong Baum-Connes property} if $\langle\CI_{\widehat{\HH}_{q}}\rangle$ is equal to the whole category $KK^{\widehat{\HH}_{q}}$.

We will prove this property in the case of a torsion-free $\G$. The reason for this assumption is that there is only one non-trivial torsion action, which can moreover be nicely described. Let $S$ be the subset of irreducible representations of $\G\ast SU_{q}(2)$ generated by the action of $\Irr(\HH_{q})$ on $u^{1}$ and let us consider inside $C^{*}(\G\ast SU_{q}(2))$ the closure $A_{q}$ of the direct sum of coefficients of representations in $S$. Then, $C^{*}(\G\ast SU_{q}(2))$ splits as a $\widehat{\HH}_{q}$-C*-algebra into a direct sum of $A_{q}$ and copies of $C^{*}(\HH_{q})$ by Theorem \ref{thm:torsionwreathproduct}. We first want to indentify $A_{q}$ as a crossed product, i.e. an image under Baaj-Skandalis duality.

\begin{lem}\label{lem:canonicaltorsion}
Let $(M_{2}(\C), \alpha_{q})$ be the torsion action of $\HH_{q}$ defined in Subsection \ref{subsec:wreathproduct}. Then, $A_{q}$ is equivariantly Morita equivalent to $\HH_{q}\ltimes M_{2}(\C)$.
\end{lem}

\begin{proof}
By Theorem \ref{thm:torsionwreathproduct}, the image of $A_{q}$ under Baaj-Skandalis duality is the only non-trivial summand in the restriction of the trivial action of $\G\ast SU_{q}(2)$ to $\HH_{q}$, hence it is equivariantly Morita equivalent to $(M_{2}(\C), \alpha_{q})$. Taking duals again yields the result. Note that one can also produce an explicit imprimitivity bimodule by mimicking the proof of \cite[Lem 5.1]{voigt2015structure}.
\end{proof}

We are now ready to prove our main result concerning the strong Baum-Connes property.

\begin{thm}\label{thm:baumconnestorsionfree}
Let $\G$ be a torsion-free discrete quantum group satisfying the strong Baum-Connes property. Then, $\G\wr_{\ast}S_{N}^{+}$ satisfies the strong Baum-Connes property, i.e. $\left\langle\T_{\G\wr_{\ast}S_{N}^{+}}\right\rangle = KK^{\G\wr_{\ast}S_{N}^{+}}$.
\end{thm}

\begin{proof}
By monoidal equivalence, it is enough to prove the result for $\HH_{q}$ and even for $\widehat{\HH}_{q}$. So let $B$ be any $\widehat{\HH}_{q}$-C*-algebra and let $B' = \Ind_{\widehat{\HH}_{q}}^{\widehat{\G}\ast \widehat{SU}_{q}(2)}(B)$ be its induction. Because
\begin{equation*}
\left\langle\T_{\G\ast SU_{q}(2)}\right\rangle = KK^{\G\ast SU_{q}(2)}
\end{equation*}
by \cite{vergnioux2013k}
and $\G\ast SU_{q}(2)$ is torsion-free by Theorem \ref{thm:actionsfreeproduct}, $B'$ is in the localizing subcategory generated by elements of the form $C^{*}(\G\ast SU_{q}(2))\otimes D$ where the action is only on the first factor. Using the fact that $C^{*}(\G\ast SU_{q}(2)) = A\oplus C^{*}(\HH_{q})^{\oplus \N}$, we see that $B'' = \Res^{\widehat{\G}\ast \widehat{SU}_{q}(2)}_{\widehat{\HH}_{q}}(B')$ is in $\left\langle\CI_{\widehat{\HH}_{q}}\right\rangle$. Moreover, $B$ is a direct summand in $B''$ as explained in the proof of \cite[Thm 5.2]{voigt2015structure}. Since triangulated subcategories are stable under retracts, $B$ is in $\left\langle\CI_{\widehat{\HH}_{q}}\right\rangle$ and we conclude by Baaj-Skandalis duality.
\end{proof}

\begin{rem}
As indicated to us by the referee, the idea of the proof (which comes from \cite[Thm 5.2]{voigt2015structure}) yields a more general statement : the strong Baum-Connes property passes to discrete quantum subgroups. Indeed, if $C(\HH)\subset C(\G)$ and if $(A, \alpha)$ is a torsion action of $\G$, we can consider the corresponding module C*-category $\CC$ over $\Rep(\G)$, which is cofinite and connected. Since cofiniteness passes to subcategories, it follows that the restriction of $\CC$ to $\Rep(\HH)$ splits as a sum of cofinite connected module C*-categories. By \cite[Lem 3.11]{arano2015torsion}, any such C*-category comes from a torsion action so that we have proved that $\CI$ is preserved under the restriction functor and this is the only fact we used in the proof if Theorem \ref{thm:baumconnestorsionfree}.
\end{rem}

We can also get another structure result for $\G\wr_{*}S_{N}^{+}$, called $K$-amenability but this requires some terminology.

\begin{de}
A $\G$-C*-algebra $A$ is said to be \emph{proper almost homogeneous} if it is equivariantly Morita equivalent to $\G\ltimes B$ for some torsion action $B$ of $\G$.
\end{de}

In particular, all elements of $\CI_{\widehat{\HH}_{q}}$ are proper almost homogeneous tensored by a trivial action by Lemma \ref{lem:canonicaltorsion}. This property passes to the generated subcategory, that is to say to all of $KK^{\widehat{\HH}_{q}}$ by Theorem \ref{thm:baumconnestorsionfree}. Because the action of a discrete quantum group on a proper almost homogeneous C*-algebra is amenable by \cite[Lem 4.6]{voigt2015structure}, we conclude that $\widehat{\HH}_{q}$ is $K$-amenable. The property of being proper almost homogeneous is stable under monoidal equivalence, thus the same reasoning works for $\G\wr_{\ast} S_{N}^{+}$, yielding

\begin{cor}
Let $\G$ be a torsion-free compact quantum group satisfying the strong Baum-Connes property. Then, the dual of $\G\wr_{\ast}S_{N}^{+}$ is $K$-amenable.
\end{cor}

Note that because all our arguments carry through monoidal equivalence, the result holds for any free wreath product by an arbitrary quantum automorphism group of a finite-dimensional C*-algebra with a distinguished state in the sense of \cite{fima2015free}.

\section{K-theory computations}\label{sec:ktheory}

From now on, $\G$ denotes a torsion-free compact quantum group with the strong Baum-Connes property and we set $G = \G\ast SU_{q}(2)$. We will use the previous results to compute the K-theory of the C*-algebras associated to the compact quantum group $\HH_{q} = \G\wr_{\ast} SO_{q}(3)$. The point here is that the corresponding C*-algebra sits inside $C(G)$ so that we can use restriction to produce an exact sequence for the K-theory of $\HH_{q}$ from one for $G$. More precisely, assume that we have a projective resolution of $\C$ of length one in $KK^{\widehat{G}}$
\begin{equation*}
0 \longrightarrow B_{1} \longrightarrow B_{0} \longrightarrow \C \longrightarrow 0.
\end{equation*}
The general theory of \cite{meyer2006baum} then implies that the Dirac element $\widetilde{\C}$ sits in an exact triangle
\begin{equation}
B_{1} \longrightarrow B_{0} \longrightarrow \widetilde{\C} \longrightarrow \Sigma B_{1}.
\end{equation}
The image of this triangle in $KK^{\widehat{\HH}_{q}}$ is still exact, hence yields a six-term exact sequence in K-theory, from which we can compute the K-theory groups of $C^{*}(\widehat{\HH}_{q}) = C(\HH_{q})$.

This strategy was used in \cite[Sec 5]{voigt2012quantum} for quantum automorphism groups of matrices, which is the case where $\G$ is trivial. In the general case, the computations can be split into two parts, one of which only depends on $SU_{q}(2)$ and is very similar to those of \cite[Sec 5]{voigt2012quantum}. We will therefore deal with this first part separately.

\subsection{Preliminary computations}

We denote by $u$ the fundamental representation of $SU_{q}(2)$ and by $p_{u}$ the minimal central projection associated to $u$ in $C_{0}(\widehat{G})$. Define, following \cite{voigt2011baum}, a KK-theory element $T_{u}\in KK^{\widehat{\HH}_{q}}(C_{0}(\widehat{G}), C_{0}(\widehat{G}))$ as the composition of
\begin{equation*}
(\id\otimes p_{u})\circ\widehat{\D} : C_{0}(\widehat{G})\to C_{0}(\widehat{G})\otimes C_{0}(\widehat{G})_{u}
\end{equation*}
with the canonical equivariant Morita equivalence $C_{0}(\widehat{G})\otimes C_{0}(\widehat{G})_{u}\sim C_{0}(\widehat{G})$. More explicitly, this is the Kasparov module
\begin{equation*}
{}_{C_{0}(\widehat{G})}\left(C_{0}(\widehat{G})\otimes H_{u}\right)_{C_{0}(\widehat{G})}
\end{equation*}
where the right action is by multiplication on the first tensor and the left action is through
\begin{equation*}
(\id\otimes p_{u})\circ\widehat{\D} : C_{0}(\widehat{G})\to C_{0}(\widehat{G})\otimes B(H_{u}).
\end{equation*}
We want to understand the image of $T_{u}$ under the descent morphism
\begin{equation*}
j : KK^{\widehat{\HH}_{q}}(C_{0}(\widehat{G}), C_{0}(\widehat{G}))\to KK(\widehat{\HH}_{q}\ltimes C_{0}(\widehat{G}), \widehat{\HH}_{q}\ltimes C_{0}(\widehat{G})).
\end{equation*}

Let us first decompose $C_{0}(\widehat{G})$ as an $\widehat{\HH}_{q}$-algebra. For a word $w$ on $\Irr(\G)\setminus\{\varepsilon_{\G}\}\sqcup\Irr(SU_{q}(2))\setminus\{\varepsilon_{SU_{q}(2)}\}$ which is either empty or starts in $\Irr(\G)$, we denote by $J_{w}$ the set of irreducible representations which are contained in a tensor product of an irreducible representation of $\HH_{q}$ by $w$ and we write
\begin{equation*}
A_{w} = \bigoplus_{x\in J_{w}}B(H_{x}),
\end{equation*}
where $\bigoplus$ is a $c_{0}$-sum. According to the results of subsection \ref{subsec:wreathproduct}, $A_{w}$ is equivariantly Morita equivalent to $C_{0}(\widehat{\HH}_{q})$. In fact, it is even isomorphic to $C_{0}(\widehat{\HH}_{q})\otimes B(H_{w})$ with the trivial action on the second factor. Moreover,
\begin{equation*}
C_{0}(\widehat{G}) = A_{u}\oplus\bigoplus_{w} A_{w},
\end{equation*}
where $A_{u}$ is the sum of the blocks obtained by tensoring representations of $\HH_{q}$ by $u$. The KK-element $T_{u}$ therefore splits as a direct sum and we will study each summand separately. Note that the crossed-product of each $A_{w}$ by $\widehat{\HH}_{q}$ is isomorphic to the compact operators. The same holds for $A_{u}$ since the crossed product, being a non-trivial torsion action, must be equivariantly Morita equivalent to $M_{2}(\C)$, hence Morita equivalent to $\C$. Thus, after splitting and applying the descent morphism to each component, the result can be described as a matrix acting on a free $\Z$-module of infinite rank. Let us denote by $e_{w}$ the unit of the copy of $\Z$ corresponding to $A_{w}$ and by $e_{u}$ the one corresponding to $A_{u}$.

\begin{prop}\label{prop:descentsuq}
The element $\partial_{u} = j(T_{u})$ acts as follows on the basis :
\begin{itemize}
\item $\partial_{u}(e_{wu^{k}}) = e_{wu^{k+1}} + e_{wu^{k-1}}$ for $w\neq\emptyset$ ending in $\Irr(\G)$,
\item $\partial_{u}(e_{w}) = e_{wu}$ for $w\neq\emptyset$ ending in $\Irr(\G)$,
\item $\partial_{u}(e_{\emptyset}) = 2e_{u}$ and $\partial_{u}(e_{u}) = 2e_{\emptyset}$.
\end{itemize}
\end{prop}

\begin{proof}
If $a\in B(H_{x})$, then $(\id\otimes p_{u})\circ\widehat{\D}(a)$ belongs the sum of the tensor products $B(H_{y})\otimes B(H_{u})$ with $y$ such that $x\subset y\otimes u$. We therefore have three cases :
\begin{itemize}
\item If $x\in J_{wu^{k}}$, then $y \in J_{wu^{k+1}}$ or $y\in J_{wu^{k-1}}$.
\item If $x\in J_{w}$ with $w\neq\emptyset$ ending in $\Irr(\G)$, then $y\in J_{wu}$.
\item If $x\in J_{u}$ then $y\in J_{\emptyset}$ and if $x\in J_{\emptyset}$ then $y\in J_{u}$.
\end{itemize}
Let us concentrate on the first case. We have to understand the module
\begin{equation*}
{}_{\widehat{\HH}_{q}\ltimes C_{0}(\widehat{\HH}_{q})\otimes B(H_{wu^{k}})}\left(\widehat{\HH}_{q}\ltimes C_{0}(\widehat{\HH}_{q})\otimes B(H_{wu^{k+1}})\otimes H_{u}\right)_{\widehat{\HH}_{q}\ltimes C_{0}(\widehat{\HH}_{q})\otimes B(H_{wu^{k+1}})}
\end{equation*}
where the right action is multiplication on the first two tensors and the left action is multiplication for the crossed-product and $(p_{wu^{k+1}}\otimes p_{u})\circ\widehat{\D}$ for $B(H_{wu^{k}})$. The crossed-product appearing here is by Takesaki-Takai duality isomorphic to the compact operators so that we can remove it. Applying then Morita equivalence on the right action, we get
\begin{equation*}
{}_{B(H_{wu^{k}})}(B(H_{wu^{k+1}})\underset{B(H_{wu^{k+1}})}{\bigotimes}H_{wu^{k+1}}\otimes H_{u})_{\C} = {}_{B(H_{wu^{k}})}(H_{wu^{k+1}}\otimes H_{u})_{\C},
\end{equation*}
with the left action given simply by the embedding of $B(H_{wu^{k}})$ as a corner in $B(H_{wu^{k+1}}\otimes H_{u})$. Applying now Morita equivalence on the left action yields
\begin{equation*}
{}_{\C}(H_{wu^{k}}\underset{B(H_{wu^{k}})}{\bigotimes}(H_{wu^{k+1}}\otimes H_{u}))_{\C}.
\end{equation*}
Recall that $H_{wu^{k+1}}\otimes H_{u} = H_{wu^{k}}\oplus H_{wu^{k+2}}$ and that the tensor product with the second component vanishes since $B(H_{wu^{k}})$ acts on it by $0$. We are therefore left with
\begin{equation*}
H_{wu^{k}}\underset{B(H_{wu^{k}})}{\bigotimes}H_{wu^{k}}
\end{equation*}
which is $1\in \Z = KK(\C, \C)$. The same computation works for the term $wu^{k-1}$.

If $w\neq\emptyset$ ends in $\Irr(\G)$, the previous computation still works and is even simpler since there is only one possible component for the right-hand side. As for the third case, it is clear that $e_{\emptyset}$ and $e_{u}$ are exchanged. Moreover, the action of $j(T_{u})$ on $\Z e_{\emptyset}\oplus\Z e_{u}$ does not depend on the compact quantum group $\G$ so that it is enough to do the computation when $\G$ is trivial. In that case, $\HH_{q} = SO_{q}(3)$ and the result was proven in \cite[Sec 5]{voigt2012quantum}.
\end{proof}

We will need later on to know the image of the linear map $d_{u} = \partial_{u} - 2\id$, which is easily computed. For convenience, let us set $E_{w} = \Span\{e_{wu^{k}} \mid k\in \N\}$ for $w\neq\emptyset$ and $E_{u} = \Z e_{\emptyset}\oplus \Z e_{u}$.

\begin{lem}\label{lem:imagesu2}
The image of $E_{u}$ under $d_{u}$ is $2\Z(e_{\emptyset} - e_{u})$. Moreover, let $(a_{k})_{k\in \N}$ be the sequence defined by $a_{0} = 2$, $a_{1} = 3$ and $a_{k+1} = 2a_{k} - a_{k-1} $. Then, the image of $E_{w}$ is the free module spanned by the vectors $\epsilon_{k} = e_{wu^{k+1}} - a_{k}e_{w}$.
\end{lem}

\begin{proof}
The first point is clear. For the second point, we will proceed by induction, showing that the image of the span of $(d_{u}(e_{wu^{l}}))_{l}$ for $0\leqslant l\leqslant k$ equals the span of $(\epsilon_{l})_{l}$. For $k=0$, we indeed have $d_{u}(e_{w}) = \epsilon_{0}$ and for $k = 1$ a straightforward computation yields 
\begin{equation*}
d_{u}(e_{wu}) = e_{wu^{2}} - 3e_{w} - 2\epsilon_{0} = \epsilon_{1} - 2\epsilon_{0}.
\end{equation*}
If we now assume $H_{k}$ for $k\geqslant 1$, we get
\begin{align*}
d_{u}(e_{wu^{k+1}}) & = e_{wu^{k+2}} + e_{wu^{k}} - 2e_{wu^{k+1}} \\
& = e_{wu^{k+2}} + (\epsilon_{k-1} + a_{k-1}e_{w}) - 2(\epsilon_{k} + a_{k}e_{w}) \\
& = e_{wu^{k+2}} - (2a_{k} - a_{k-1})e_{w} + \epsilon_{k-1} - 2\epsilon_{k} \\
& = \epsilon_{k+1} + \epsilon_{k-1} - 2\epsilon_{k}
\end{align*}
Since the family $(\epsilon_{k})_{k}$ is free, the result follows.
\end{proof}

\subsection{Free orthogonal quantum groups}

The first computation will be for $\G = O_{n}^{+}$. Let us denote by $v$ its fundamental representation and let $T_{v}\in KK^{\widehat{\HH}_{q}}(C_{0}(\widehat{G}), C_{0}(\widehat{G}))$ be defined in the same way as $T_{u}$. We will denote its image under the descent morphism by $\partial_{v}$ and set $d_{v} = \partial_{v} - n\id$. Consider the homological ideal $\mathfrak{J} = \ker(\Res : KK^{\widehat{G}}\to KK)$. It was proven in \cite{vergnioux2013k} that the complex
\begin{equation}
0 \longrightarrow C_{0}(\widehat{G})^{\oplus 2} \overset{\delta}{\longrightarrow} C_{0}(\widehat{G}) \overset{\lambda}{\longrightarrow} \C \longrightarrow 0,
\end{equation}
where $\delta = (T_{u} - \dim(u)\id)\oplus (T_{v} - \dim(v)\id)$ and $\lambda$ is the regular representation, is a $\mathfrak{J}$-projective resolution in $KK^{\widehat{G}}$. It therefore yields an exact triangle in $KK^{\widehat{\HH}_{q}}$ containing $\widetilde{\C}$. Applying the functor $K(\widehat{\HH}_{q}\ltimes\cdot)$ leads to the six-term exact sequence
\begin{equation*}
\xymatrix{ \left(\bigoplus_{w}\Z_{w}\right)\oplus\Z_{u} \ar[r] & K_{0}(C(\HH_{q})) \ar[r] & 0 \ar[d] & \\
 \left(\left(\bigoplus_{w}\Z_{w}\right)\oplus\Z_{u}\right)^{\oplus 2} \ar[u]^{d_{u}\oplus d_{v}}& K_{1}(C(\HH_{q})) \ar[l]& 0 \ar[l] }
\end{equation*}
and the K-theory groups are given by the kernel and coimage of $d_{u}\oplus d_{v}$. To compute these spaces, we need an analogue of Proposition \ref{prop:descentsuq} for $\partial_{v}$. The map $d_{u}\oplus d_{v}$ acts on two copies of $(\oplus_{w}\Z_{w})\oplus\Z_{u}$ and for convenience we will denote by $(f_{w})_{w}$ the basis of the second copy.

\begin{prop}\label{prop:descentorthogonal}
The element $\partial_{v} = j(T_{v})$ acts as follows on the basis :
\begin{itemize}
\item $\partial_{v}(f_{wv^{k}}) = f_{wv^{k+1}} + f_{wv^{k-1}}$,
\item $\partial_{v}(f_{w}) = f_{wv}$ for $w$ ending in $\Irr(SU_{q}(2))$ or $w = \emptyset$.
\item $\partial_{v}(f_{u}) = nf_{u}$.
\end{itemize}
\end{prop}

\begin{proof}
The first two cases follow from the same computations as in Proposition \ref{prop:descentsuq}. As for the third one, first note that the only representation $y$ such that $u\subset y\otimes v$ is $uv$ and that $uv\in J_{u}$ since $uv\subset uvu\otimes u$. Thus, $\partial_{v}(f_{u})\in \Z f_{u}$. Before applying the descent map $j$, we are considering the Kasparov module
\begin{equation*}
{}_{A_{u}}(A_{u}\otimes H_{v})_{A_{u}}
\end{equation*}
with the left action given by $(\id\otimes p_{v})\circ\widehat{\D}$ and the right one given by multiplication on the first tensor. Taking crossed-products we get
\begin{equation*}
{}_{\widehat{\HH}_{q}\ltimes A_{u}}(\widehat{\HH}_{q}\ltimes A_{u}\otimes H_{v})_{\widehat{\HH}_{q}\ltimes A_{u}}
\end{equation*}
By Lemma \ref{lem:canonicaltorsion}, this yields the same KK-theory element as
\begin{equation*}
{}_{M_{2}(\C)}(M_{2}(\C)\otimes H_{v})_{M_{2}(\C)}
\end{equation*}
where the right action is multiplication on the first tensor and all we have to know on the left action is that it is faithful. Applying Morita equivalence on the right, we get
\begin{equation*}
{}_{M_{2}(\C)}(\C^{2}\otimes H_{v})_{\C} = {}_{\C}([\C^{2}\otimes H_{v}] \underset{M_{2}(\C)}\otimes\C^{2})_{\C}.
\end{equation*}
By faithfulness of the representation, the latter Hilbert space has dimension $n$, hence the result.
\end{proof}

The problem now boils down to linear algebra in $\Z$-modules. For clarity we first determine the kernel.

\begin{lem}\label{lem:kernelsumorthogonal}
The kernel of $d_{u}\oplus d_{v}$ is $\Z(e_{\emptyset} + e_{u})\oplus\Z f_{u}$.
\end{lem}

\begin{proof}
Consider an element
\begin{equation*}
x = \lambda_{u}e_{u} + \mu_{u}f_{u} + \sum_{w}\lambda_{w}e_{w} + \mu_{w}f_{w}
\end{equation*}
and let $L$ be the maximal length of the words appearing with non-zero coefficient in this sum. If we assume that $L > 1$, then a word $w$ of length $L$ is either of the form $w'u^{k}$ or $w'v^{k}$. In the first case, let $k_{0}$ be maximal so that $w = w'u^{k_{0}}$ occurs in $x$. Then,
\begin{equation*}
\lambda_{w}e_{w'u^{k_{0}+1}} - (d_{u}\oplus d_{v})(x)
\end{equation*}
must be a linear combination of basis vectors not including $e_{wu^{k_{0}+1}}$, which is impossible if $(d_{u}\oplus d_{v})(x) = 0$. Similarly, if $\mu_{w}\neq 0$ then $\mu_{w}f_{wv} - (d_{u}\oplus d_{v})(x)$ yields a contradiction. The same argument works if $w = w'v^{k}$. If we assume now that $L = 1$ and $w = v^{k}$ we get the same contradiction. In conclusion, the sum only contains terms associated to $w = u$ or $w = \emptyset$ and
\begin{equation*}
x = \lambda_{u} e_{u} + \mu_{u}f_{u} + \lambda_{\emptyset}e_{\emptyset} + \mu_{\emptyset}f_{\emptyset}.
\end{equation*}
Then,
\begin{align*}
(d_{u}\oplus d_{v})(x) & = \lambda_{u}(2e_{\emptyset} - 2e_{u}) + 0 + \lambda_{\emptyset}(2e_{u} - 2e_{\emptyset}) + \mu_{\emptyset}(f_{v} - nf_{\emptyset}) \\
& = (2\lambda_{u} - 2\lambda_{\emptyset})e_{\emptyset} + (2\lambda_{\emptyset} - 2\lambda_{u})e_{u} + \mu_{\emptyset}(f_{v} - nf_{\emptyset}),
\end{align*}
forcing $\lambda_{u} = \lambda_{\emptyset}$ and $\mu_{\emptyset} = 0$.
\end{proof}

We can now compute the K-theory of $\HH_{q}$ for $\G = O_{n}^{+}$.

\begin{prop}
Let $\G = O_{n}^{+}$ with $n\geqslant 2$. Then, $K_{0}(C(\HH_{q})) = \Z\oplus\Z_{2}$ and $K_{1}(C(\HH_{q})) = \Z^{2}$.
\end{prop}

\begin{proof}
The computation of $K_{1}(C(\HH_{q}))$ follows from Lemma \ref{lem:kernelsumorthogonal}. To compute $K_{0}(C(\HH_{q}))$, we first have to compute the image of the two maps. For $d_{u}$ this was done in Lemma \ref{lem:imagesu2}. For $d_{v}$ a similar computation shows that the image is spanned by the vectors $\eta_{k} = e_{wv^{k+1}} - b_{k}e_{w}$ where $w$ is a (possibly empty) word ending in $\Irr(SU_{q}(2))\setminus\{\varepsilon_{SU_{q}(2)}\}$ and $b_{k}$ is defined by $b_{0} = n$, $b_{1} = n^{2}-1$ and $b_{k+1} = nb_{k} - b_{k-1}$.

Let us denote by $p$ the canonical surjection onto the coimage of $d_{u}\oplus d_{v}$. If $w = w'u^{k}$ is a word of length at least two, then
\begin{equation*}
p(e_{w}) = a_{k-1}p(e_{w'})
\end{equation*}
by Lemma \ref{lem:imagesu2}. Similarly, if $w = w'v^{k}$ then $p(e_{w}) = b_{k-1}p(e_{w'})$. Thus, the coimage is spanned by the vectors $p(e_{w})$ for $w$ of length at most one. Since $p(e_{v^{k}}) = b_{k-1}p(e_{\emptyset})$, we even have to consider only $p(e_{\emptyset})$ and $p(e_{u})$ so that the coimage has rank at most two. Since $e_{u}$ is not in the image of $d_{v}$, the only relation between $p(e_{\emptyset})$ and $p(e_{u})$ is the one coming from $d_{u}$, namely $2p(e_{\emptyset}) = 2p(e_{u})$. Hence, the coimage is $\Z\oplus\Z_{2}$ generated by $p(e_{\emptyset} + e_{u})$ and $p(e_{\emptyset} - e_{u})$.
\end{proof}

\subsection{Free unitary quantum groups}

The previous result can easily be extended to free unitary quantum groups and even to arbitrary free products of free orthogonal and free unitary quantum groups. Indeed, let
\begin{equation*}
\G = U_{P_{1}}^{+}\ast\cdots\ast U_{P_{k}}^{+}\ast O_{Q_{1}}^{+}\ast\cdots\ast O_{Q_{l}}^{+}
\end{equation*}
where $P_{i}\in GL_{m_{i}}(\C)$ with $m_{i}\geqslant 2$ and $Q_{j}\in GL_{n_{j}}(\C)$ with $n_{j}\geqslant 2$ satisfies $Q_{j}\overline{Q_{j}} = \pm\id$. It is known from \cite{vergnioux2013k} that if $v_{i}, u_{j}$ denote the fundamental representations of $U_{P_{j}}^{+}$ and $O_{Q_{j}}^{+}$ respectively, then
\begin{equation*}
\delta = \left(\bigoplus_{i}(T_{v_{i}} - \dim(v_{i})\id)\oplus (T_{\overline{v}_{i}}-\dim(\overline{v}_{i})\id)\right)\oplus\left(\bigoplus_{j}T_{u_{j}} - \dim(u_{j})\id\right)
\end{equation*}
yields a projective resolution
\begin{equation*}
0 \longrightarrow C_{0}(\widehat{G})^{\oplus 2k+l} \overset{\delta}{\longrightarrow} C_{0}(\widehat{G}) \overset{\lambda}{\longrightarrow} \C \longrightarrow 0.
\end{equation*}

The same computations as before then yield

\begin{prop}
Let $\G$ be a free product of free orthogonal and unitary quantum groups as above. Then, $K_{0}(C(\HH_{q})) = \Z\oplus \Z_{2}$ and $K_{1}(C(\HH_{q})) = \Z^{2k+l+1}$.
\end{prop}

\subsection{Classical free groups}

Our third and last example will be free groups $\G = \widehat{\F}_{n}$ on finitely many generators. The case $n=1$ is particularly interesting since the corresponding free wreath product $\widehat{\Z}\wr_{\ast}S_{N}^{+}$ is isomorphic to the quantum reflection group $H_{N}^{\infty +}$. The resolution given in \cite{vergnioux2013k} is not valid if we replace $U_{n}^{+}$ by a free groups, but it is easy to modify it. Keeping the previous notations and denoting by $a_{1}, \cdots, a_{n}$ the generators of $\F_{n}$, we consider the complex
\begin{equation*}
0 \longrightarrow C_{0}(\widehat{G})^{\oplus n+1} \overset{\delta}{\longrightarrow} C_{0}(\widehat{G}) \overset{\lambda}{\longrightarrow} \C \longrightarrow 0.
\end{equation*}
where $\delta = (T_{a_{1}} - \id)\oplus\cdots\oplus (T_{a_{n}} - \id)\oplus (T_{u} - 2\id)$. We need to prove that this is a $\mathfrak{J}$-projective resolution, where $\mathfrak{J}$ is the kernel of the restriction functor $KK^{\widehat{\F}_{N}\ast SU_{q}(2)}\to KK$. By definition, it is enough to check exactness when applying the functor $KK^{\widehat{\F}_{N}\ast SU_{q}(2)}(C_{0}(\widehat{\F}_{N}\ast SU_{q}(2)), \cdot)$. This yields the diagram
\begin{equation}\label{eq:integersfreeproduct}
0 \longrightarrow R_{G}^{\oplus n+1} \overset{d}{\longrightarrow} R_{G} \overset{\varepsilon}{\longrightarrow} \Z \longrightarrow 0
\end{equation}
with $d = d_{a_{1}^{-1}}\oplus \cdots\oplus d_{a_{n}^{-1}}\oplus d_{u}$ and $\varepsilon$ is the map induced by the dimension function.

\begin{lem}
The diagram \eqref{eq:integersfreeproduct} is an exact sequence.
\end{lem}

\begin{proof}
The surjectivity of $\varepsilon$ is clear, let us prove that $d$ is injective. We will denote by $e_{w}^{k}$ the basis element corresponding to the representation $w$ in the $k$-th copy of $R_{G}$. Assume that $x$ is a finite linear combination of such elements in the kernel of $d$ and let $w$ be a word of maximal length appearing in $x$. If it appears in the last component, then either $w$ ends in $\F_{n}$ and its image under $d$ contains $wu$, or it ends with some $u^{k}$ and taking $k$ maximal, its image under $d$ contains the same word but ending with $u^{k+1}$. In both, cases, we get a contradiction by the same argument is in Lemma \ref{lem:kernelsumorthogonal}. We may therefore assume that $w$ appears in one of the first $n$ components, say the first one (the other cases being similar). If $w$ ends with $u^{k}$, then again we get a longer word after applying $d$ which cannot be simplified. Thus, $w = w'\gamma$ for some $\gamma\in F_{n}$ and we may assume that the word length of $\gamma$ is maximal. Then, still by the same argument $\gamma$ must end with $a_{1}$ otherwise we would get a longer word. But then, $d(w\gamma) = w\gamma a_{1}^{-1} - w\gamma$ so that $w\gamma$ appears in $d(x)$. This means that there must be terms in $x$ whose images simplifies with $w\gamma$, but because we are working with free groups such a term must be of the form $w(\gamma a_{i})$, contradicting the maximality of the word length of $\gamma$. As a conclusion, $d$ is injective.

The vectors $x_{w} = w - \dim(w)$ for a basis of the kernel of $\varepsilon$ and we now have to compute their image in the quotient by the image of $d$ (which is clearly contained in $\ker(\varepsilon)$). Let us consider a non-empty word $w$. If $w = w'u^{k}$, the equality
\begin{equation*}
w'u^{k} = d_{u}(w'u^{k-1}) - w'u^{k-2} + 2w'u^{k-1}
\end{equation*}
(with the convention $w'u^{-1} = 0$) yields
\begin{align*}
x_{w} & = d_{u}(w'u^{k-1}) - p(x_{w'u^{k-2}}) + 2p(x_{w'u^{k-1}}) - \dim(w'u^{k-2}) + 2\dim(w'u^{k-1}) - \dim(w'u^{k}) \\
& = d_{u}(w'u^{k-1}) - p(x_{w'u^{k-2}}) + 2p(x_{w'u^{k-1}})
\end{align*}
so that $p(x_{w}) \in \Z p(w'u^{k-2})\oplus\Z p(w'u^{k-1})$. Applying this inductively, we see that $p(w)\in \Z p(w')$. If $w = w'a_{l}^{k}$, then using
\begin{equation*}
w'a_{l}^{k} - w'a_{l}^{k+1} = d_{a_{l}^{-1}}(w'a_{l}^{k+1})
\end{equation*}
we see that we can increase or decrease $k$ depending on its sign until we get $p(w)\in \Z p(w')$. We have shown that all the basis elements have the same image. Since $x_{u}$ is the image of the trivial representation, its image is $0$ and the proof is complete.
\end{proof}

We can now apply the same strategy as before : we restrict the resolution to $\widehat{\HH}_{q}$ and compute the corresponding six-term exact sequence. This yields

\begin{prop}
Let $\G = \widehat{\F}_{n}$. Then, $K_{0}(C(\HH_{q})) = \Z\oplus \Z_{2}$ and $K_{1}(C(\HH_{q})) = \Z^{n+1}$.
\end{prop}

If $\Gamma$ is a torsion-free group which is not free, then the above sequence does not work anymore. Indeed, any relation in $\Gamma$ will produce elements in the kernel of $d$ so that its injectivity fails. For instance if we consider $\Z^{2}$ instead of $\F_{2}$, then
\begin{equation*}
d_{a^{-1}}(a - ab) = d_{b^{-1}}(aba^{-1} - ab).
\end{equation*}
One way to go round this problem could be to build a longer projective resolution. The price to pay then is that the exact sequence in K-theory will not be obtained directly but through a spectral sequence computation.

\bibliographystyle{amsplain}
\bibliography{Freslon_Martos}

\end{document}